\documentclass[11pt]{amsart}
\usepackage{amsmath,amssymb,amscd,verbatim,color}

\newtheorem{Theorem}{Theorem}[section]
\newtheorem{Lemma}{Lemma}[section]
\newtheorem{Proposition}{Proposition}[section]

\numberwithin{equation}{section}

\def\bc{\begin{center}}
\def\ec{\end{center}}

\def\a1{(a_1, a_2, \cdots, a_n)}

\allowdisplaybreaks

\begin{document}

\title[Asymptotic pairs, stable sets and chaos]
{Asymptotic pairs, stable sets and chaos in positive entropy
systems}

\author[W. Huang]{Wen Huang}
\address{W. Huang: Wu Wen-Tsun Key Laboratory of Mathematics, USTC, Chinese Academy of Sciences, Hefei Anhui 230026, PRC}
\email{wenh@mail.ustc.edu.cn}
\author[L. Xu]{Leiye Xu}
\address{L. Xu: School of Mathematics, Jilin University,
Changchun, 130012, PRC} \email{xuly12@mails.jlu.edu.cn}

\author[Y. Yi]{Yingfei Yi}
\address{Y. Yi: School of Mathematics, Jilin University,
Changchun, 130012, PRC, and School of Mathematics, Georgia Institute
of Technology, Atlanta, GA 30332, USA} \email{yi@math.gatech.edu}

\thanks {The first author was partially supported by NNSF for Distinguished
Young Scholar (11225105). The third author was partially supported
by  NSF grant DMS1109201 and a Scholarship from Jilin
University.}



\begin{abstract}  We consider positive entropy
$G$-systems for certain
 countable, discrete, infinite left-orderable amenable groups $G$.
 By undertaking  local analysis, the existence  of asymptotic pairs and chaotic sets
 will be studied in connecting with the  stable sets. Examples are
 given for the case of integer lattice groups, the Heisenberg group,
 and the groups of integral unipotent upper triangular
matrices.
\end{abstract}
\maketitle

\section{Introduction}
Throughout the paper,  we let  $G$ be a countable,
discrete, infinite amenable group.  Recall that a countable discrete
group $G$ is said to be {\it amenable} if there exists an invariant
mean on it,  or equivalently,  if there exists a
sequence of finite subsets $F_n\subset   G$, such that,  for every
$g \in G$,
\begin{equation}\label{cond-1}
\lim_{n\rightarrow +\infty} \frac{|gF_n\Delta F_n|}{|F_n|}=0.
\end{equation}
A sequence satisfying condition \eqref{cond-1} is called a {\it
F{\o}lner sequence} (see \cite{EF}).  For instance,
when $G=\mathbb{Z}^d$ for  some $d\in \mathbb{N}$,
$\{F_n=[0,n]^d: n\in \mathbb N\}$ is a F{\o}lner sequence of $G$.

A {\it $G$-system} $(X,G)$  is such that $X$ is a
compact metric space and $G$ continuously acts on $X$. Let $(X,G)$
be a $G$-system and $S$ be an infinite subset of $G$, where $X$ is
endowed  with  the metric $d$. A pair $(x,y)\in X\times
X$ is called  a {\it $S$-asymptotic pair} if for each
$\epsilon>0$, there are only  finitely many elements
$s\in S$ with $d(sx,sy)>\epsilon$. For a point $x\in
X$, the set
$$W_S(x,G)=\{y\in X: (x,y) \text{ is an $S$-asymptotic pair}\}$$
 is called the {\it $S$-stable set} of $x$.  Let
$\delta>0$. A pair $(x,y)\in X\times X$ is  called a
{\it $(S,\delta)$-Li-Yorke pair} if $(x,y)$ is $S'$-asymptotic pair
for some infinite subset $S'$ of $S$ and $\{s\in
S:d(sx,sy)>\delta\}$ is an infinite subset of $S$. A subset $E$ of
$X$ is said to be {\it $(S,\delta)$-chaotic} if
$(x,y)$ is  a $(S,\delta)$-Li-Yorke pair for every
$x\neq y\in E$.

Given a $G$-system $(X,G)$, one  can define in the usual
way its topological entropy $h_{\text{top}}(G,X)$ as well as
measure-theoretic entropy  with respect to an invariant Borel
probability measure, lying in $[0,+\infty]$ (see Section 2 for
detail).  One fundamental question to ask is the
relationship between  the positivity of  the entropies
of $(X,G)$ and the chaotic behavior of  its orbits. A
well-known result in this direction is  given by
Blanchard  and co-authors in \cite{BGKM}
for the case $G=\mathbb{Z}$.  They showed  via
ergodic-theoretic method that if a $\mathbb{Z}$-system $(X,\mathbb{Z})$ has positive topological entropy, then there
exists an uncountable $(\mathbb{Z}_+,\delta)$-chaotic subset $E$ of
$X$ for some $\delta>0$, where $\mathbb{Z}_+=\{0,1,2,\cdots\}$
 (see also \cite{KL} for an alternative proof using
combinatorial method). The result  in \cite{BGKM} is
later generalized to the case of  sofic group actions by Kerr and Li
(\cite[Corollary 8.4]{KL1}).

 Another fundamental question to ask is the relationship
between the positivity of the  entropy of a $G$-system $(X,G)$
and the existence of asymptotic pairs.  On one hand,
Lind and Schmidt (\cite[Example 3.4]{LS}) constructed examples of
$\mathbb{Z}$-actions generated by toral automorphisms which have
positive entropy but admit no non-diagonal $\mathbb{Z}$-asymptotic
pairs. In fact, by using co-induction techniques introduced in
\cite{DGRS,DZ}, such examples can be constructed for general
$G$-actions if $G$ is an infinite countable discrete amenable group
containing a subgroup isomorphic to $\mathbb{Z}$. Thus,
additional conditions are needed for a positive entropy $G$-system
to admit non-diagonal $G$-asymptotic pairs. Indeed, it was shown by
Schmidt (see \cite[Proposition 2.1]{Sc}) for the case $G=
\mathbb{Z}^d$ for some $d \in\mathbb{N}$ that every subshift of
finite type with positive entropy has non-diagonal $G$-asymptotic
pairs. For an expansive action of $\mathbb{Z}^d$ by (continuous)
automorphisms of a compact abelian group, Lind and Schmidt \cite{LS}
proved that the action has positive entropy if and only if there
exist non-diagonal $\mathbb{Z}^d$-asymptotic pairs. Recently, Chung
and Li \cite{CL} extended this result in \cite{LS} for a larger
class of amenable groups.

 On the other hand, it was shown by Blanchard et al in
\cite{BHR}  for the case $G = \mathbb{Z}$  that if the
$\mathbb{Z}$-system is generated by a homeomorphism $T :
X\rightarrow  X$  and if $\mu$ is a $T$-invariant ergodic Borel
probability measure on $X$ with positive entropy, then there exists
$\delta>0$ such that for $\mu$-a.e. $x\in X$, there exists an
uncountable subset $F_x\subseteq W_{\mathbb{Z}_+}(x,\mathbb{Z})$
 with the property that $(x,y)$ is a
$(\mathbb{Z}_-,\delta)$-Li-Yorke pair for each $y\in F_x$, where
$\mathbb{Z}_-=\{ 0,-1,-2,\cdots \}$.  We refer the
readers to \cite{Z} for relativized versions of the results  in
\cite{BHR, BGKM}, to \cite{H,HLY,S} for  more precise
characterizations on chaotic phenomenon appearing in  stable sets
of positive entropy $\mathbb{Z}$-systems,  and to
\cite{FHYZ} for  dimension analysis of  these sets.

\medskip

 With the fundamental questions above in mind, the aim
of the present paper is to investigate the connections among the
positivity of the entropies, chaotical behavior, and the existence
of asymptotical pairs in a general $G$-system when $G$ is a
countable discrete infinite amenable group with the algebraic past
$\Phi$. More precisely, let $G$ be a group with the unit $e_G$. $G$
is said to be {\it left-orderable} if there exists a linear ordering
in $G$ which is invariant under left translations.  The
group $G$ is left-orderable if and only if it contains a subset
$\Phi$,  called an {\it algebraic past} of $G$
(\cite{AV}), with the following properties:
\begin{enumerate}
\item  $\Phi\cap\Phi^{-1}=\emptyset,$
\item  $\Phi\cup \Phi^{-1} \cup\{e_G\}=G,$
\item  $\Phi\cdot \Phi\subseteq\Phi.$
\end{enumerate}
 Indeed, with respect to the algebraic past $\Phi$, one
obtains the desired linear ordering  on $G$ as follows: $g_1$ is
less than $g_2$ (write $g_1<_{\Phi} g_2$), if $g_2^{-1}g_1\in\Phi$.
Let $(g_i)_{i\geq 1}$ be a sequence in $G$. We say  that
a sequence $(g_i)_{i\geq1}$ {\it increasingly goes to infinity with
respect to $\Phi$} (write $g_i\nearrow \infty$ w.r.t $\Phi$) if
$g_i<_\Phi g_{i+1}$ for each $i\ge 1$ and for each element $g\in G$,
$\#\{ i\in \mathbb{N}: g_i<_\Phi g\}<+\infty$. We say
that a sequence $(g_i)_{i\geq1}$ {\it decreasingly goes to infinity
with respect to $\Phi$} (write $g_i\searrow \infty$ w.r.t $\Phi$) if
$g_i>_\Phi g_{i+1}$ for each $i\ge 1$ and  for each element $g\in
G$, $\#\{ i\in \mathbb{N}: g_i>_\Phi g\}<+\infty$.

The theory of left-orderable groups is a well developed subject in
group theory  which can be traced back to  the late
nineteenth century. For more details of this theory, we refer the
readers to widely known modern books \cite{BR}, \cite{KM}. It is
well-known that a nontrivial left-orderable group  must
be torsion-free,  and the category of left-orderable
groups includes torsion-free nilpotent groups and free groups.

 One main result of the paper concerns the existence of
asymptotic pairs in a positive entropy $G$-system as follows.

\begin{Theorem}\label{thm1}Let $G$ be a countable discrete infinite amenable group with the algebraic past $\Phi$,
$f_n\nearrow \infty$ w.r.t $\Phi$  with $f_n\Phi f_n^{-1}=\Phi$ for $n\ge 1$, and
$S$ be an infinite subset of $G$ such that $\sharp\{s\in S:s<_\Phi
f_n\}<\infty$ for each $n\ge 1$. Then any positive entropy
$G$-system has proper $S$-asymptotic pairs. More precisely, if
$(X,G)$ is a $G$-system and $\mu$ is a positive entropy ergodic
$G$-invariant Borel probability measure on $X$, then
 $W_S(x,G)\setminus \{x\}\neq \emptyset$ for $\mu$-a.e. $x\in X$.
\end{Theorem}

Let $G$ be a countable discrete infinite amenable group with the
algebraic past $\Phi$.  If $S$ is a subset of a group $G$, then
$<S>$, {\it the subgroup generated by $S$}, is the smallest subgroup
of $G$ containing $S$. A semigroup $S$ of $G$ is called {\it
$\Phi$-admissible} if $S\subset \Phi^{-1}\cup\{e_G\}$, $<S>=G$ and
there exist $f_n\nearrow \infty$ w.r.t $\Phi$ and $h_n\searrow
\infty$ w.r.t $\Phi$ such that $\sharp\{s\in S:s<_\Phi
f_n\}<\infty$, $h_n\in S^{-1}$ and $f_n\Phi f_n^{-1}=\Phi$ for each
$n\ge 1$. It is clear that a $\Phi$-admissible semigroup $S$ of $G$
 is always an infinite semigroup since $G$ is
torsion-free.

 Another  main result of the paper concerns the
existence of certain chaotic sets in a positive entropy $G$-system
as follows.

\begin{Theorem}\label{thm2} Let $G$ be a countable discrete infinite amenable
group with the algebraic past $\Phi$ and $S$ be a $\Phi$-admissible semigroup of $G$. If $(X,G)$ is a
$G$-system and $\mu$ is a positive entropy ergodic $G$-invariant
Borel probability measure on $X$, then there exists $\delta>0$ such
that
 for $\mu$-a.e. $x\in X$, the stable set $W_S(x,G)$ contains a
 $(S^{-1},\delta)$-chaotic set which is a Cantor set.
\end{Theorem}

 To demonstrate applications of Theorem~\ref{thm2}, we
give below three examples of $\Phi$-admissible semigroups $S$ for
some special left-orderable amenable groups $G$. For the first
example, we consider $G=\mathbb{Z}^d$ for some $d\in \mathbb{N}$ and
let $S=\mathbb{Z}_+^d:=\{(n_1,\cdots,n_d)\in \mathbb{Z}^d:n_i\ge 0,
1\le i\le d\}$,
\begin{align*}
 \Phi=\{&(n_1,\cdots,n_d)\in \mathbb{Z}^d: \exists j\in
\{0,1,2,\cdots,d-1\} \text{ such that } \nonumber\\
&\sum_{\ell=1}^{d-k} n_\ell=0 \text{ for }k=0,\cdots,j-1 \text{ and
}\sum_{\ell=1}^{d-j} n_\ell<0\},
\end{align*}
and $f_n=(n,n,\cdots,n)\in \mathbb{Z}^d$, $n\in
\mathbb{N}$. Then it is not difficult to see that $\Phi$ is an
algebraic past of $G$,  $<S>=G$, $f_n\nearrow \infty$ w.r.t $\Phi$,
$f_n^{-1}\searrow \infty$ w.r.t $\Phi$, and $\sharp\{s\in S:s<_\Phi
f_n\}<\infty$, $f_n^{-1}\in S^{-1}$ and $f_n\Phi f_n^{-1}=\Phi$ for
each $n\ge 1$. Thus $S$ is $\Phi$-admissible.  By
applying Theorem \ref{thm2}, we immediately obtain the following
result.

\begin{Theorem}\label{cor1} Any positive entropy $\mathbb{Z}^d$-system has proper $\mathbb{Z}^d_+$-asymptotic
pairs.
More precisely, if $(X,\mathbb{Z}^d)$ is a $\mathbb{Z}^d$-system and
$\mu$ is a positive entropy ergodic $\mathbb{Z}^d$-invariant Borel
probability measure on $X$, then there exists $\delta>0$ such that
for $\mu$-a.e. $x\in X$, the stable set
$W_{\mathbb{Z}^d_+}(x,\mathbb{Z}^d)$ contains a
$(\mathbb{Z}^d_{-},\delta)$-chaotic set which is a Cantor set, where
$\mathbb{Z}_-^d:=\{(n_1,\cdots,n_d)\in \mathbb{Z}^d:n_i\le 0, 1\le
i\le d\}$.
\end{Theorem}

 Our next example treats the case of  Heisenberg group
 -- the two-step nilpotent countable matrix group
\begin{equation}\label{heizenberg}
G =\left\{\begin{pmatrix}
1& m_3 & m_1 \\
0 & 1 & m_2 \\
0 & 0 & 1
 \end{pmatrix} : m_1,m_2,m_3\in \mathbb{Z} \right\}.
\end{equation}
We fix the generators
$$T_1=\begin{pmatrix}
1&  0 & 1 \\
0&  1 & 0 \\
0&  0 & 1
 \end{pmatrix} \ \ ;
 T_2=\begin{pmatrix}
1 & 1 & 0 \\
0 & 1 & 0 \\
0 & 0 & 1
 \end{pmatrix} \ \ ; T_3=\begin{pmatrix}
1&  0 & 0 \\
0 & 1 & 1 \\
0 & 0 & 1
 \end{pmatrix}.$$
Then $$T_3^{n_3}T_2^{n_2}T_1^{n_1}=\begin{pmatrix}
1&  n_2 & n_1 \\
0&  1 & n_3 \\
0&  0 & 1
 \end{pmatrix}  \text{ for }n_1,n_2,n_3\in \mathbb{Z}. $$
Here $T_1$ generates the center $Z$ of $G$. Define the linear order
relation on the above generators by setting $T_3 > T_2 > T_1$,
together with the associated lexicographic linear order relation on
$G$,  i.e.,
$T_3^{j_3}T_2^{j_2}T_1^{j_1}>T_3^{j_3}T_2^{j_2}T_1^{j_1}$ if and
only if $(j_3, j_2, j_1)$ is lexicographically less than $(k_3, k_2,
k_1)$. This order relation is invariant with respect to the left
translations of $G$,  so we obtain an algebraic past
$\Phi$ in $G$ defined as the subset of all elements of G which are
less than the identity $I_3$. Since $G$ is nilpotent,
it is amenable (\cite[Proposition 4.6.6]{CC}). Thus the Heisenberg group $G$ is a
countable discrete infinite amenable group with algebraic past
$\Phi$.

Let $f_n=T_3^n, h_n=T_3^{-n}$ for $n\in \mathbb{N}$ and
\begin{equation}\label{s-heizenberg}
S:=\{ T_3^{n_3}T_2^{n_2}T_1^{n_1}: (n_3,n_2,n_1)\in \mathbb{Z}^3 \text{ with }n_3\ge n_2\ge 0 \text{ and } n_3^2\ge n_1\ge 0\}.
\end{equation}
Then $S\subset \Phi^{-1}\cup \{I_3\}$ is a semigroup of $G$, $<S>=
<T_3,T_3T_2,T_3T_1>=G$ and $f_n\nearrow \infty$ w.r.t. $\Phi$ and
$h_n\searrow \infty$ w.r.t. $\Phi$. It is also not hard to see that
$$\sharp\{s\in S:s<_\Phi f_n\}\le \sharp\{(n_3,n_2,n_1): n\ge n_3\ge
n_2\ge 0 \text{ and } n_3^2\ge n_1\ge 0\}\le (n^2+1)^3<\infty,$$ $f_n\Phi
f_n^{-1}=\Phi$ and $h_n\in S^{-1}$ for $n\ge 1$.
Summarizing up, $S$ is a $\Phi$-admissible semigroup of the
Heisenberg group $G$. Therefore, an application of
Theorem \ref{thm2} yields the following result.

\begin{Theorem}\label{cor2} Let $G$ be the Heisenberg group defined
 in \eqref{heizenberg} and $S$ be the semigroup defined
 in  \eqref{s-heizenberg}. Then any positive entropy
$G$-system has proper $S$-asymptotic pairs. More precisely, if
$(X,G)$ is a $G$-system and $\mu$ is a positive entropy ergodic
$G$-invariant Borel probability measure on $X$, then there exists
$\delta>0$ such that for $\mu$-a.e. $x\in X$, the stable set
$W_{S}(x,G)$ contains a $(S^{-1},\delta)$-chaotic set which is a
Cantor set.
\end{Theorem}

 We now turn to a more general case  of the group of integral unipotent upper triangular
matrices. Given $d\in \mathbb{N}$,  consider the
matrix-valued function
$$
M({\bf a})=\left(
  \begin{array}{cccccccc}
    1 & a_1^1 & a_1^2 & a_1^3  &\ldots & a_1^{d-1} &a_1^d \\
    0 & 1     & a_2^1 & a_2^2   &\ldots &  a_2^{d-2}&a_2^{d-1}\\
    0 & 0     &1      & a_3^1   &\ldots &  a_3^{d-3} &a_3^{d-2}\\
   \vdots & \vdots & \vdots &  \vdots & \vdots & \vdots &\vdots\\
    0 & 0    &0       &0           & \ldots & a_{d-1}^1 &a_{d-1}^2 \\
    0 & 0    &0       &0            & \ldots & 1 &a_{d}^1 \\
    0 & 0    &0       &0            & \ldots & 0  & 1
  \end{array}
\right),
$$
 ${\bf a}=(a_i^k)_{1\le k\le d, 1\le i\le d-k+1}\in
\mathbb{\mathbb{Z}}^{d(d+1)/2}$. Then the group
$$G=:U_{d+1}(\mathbb{Z})=\{ M({\bf a}):{\bf a}=(a_i^k)_{1\le k\le d, 1\le i\le d-k+1}\in
\mathbb{\mathbb{Z}}^{d(d+1)/2}\}$$   of integral
unipotent upper triangular matrices is in fact a $d$-step nilpotent
group. It is clear that for  any $A\in
U_{d+1}(\mathbb{Z})$ there exists a unique ${\bf a}=(a_i^k)_{1\le
k\le d, 1\le i\le d-k+1}\in \mathbb{Z}^{d(d+1)/2}$ such that
$A=M({\bf a})$.  Moreover, for any ${\bf
a}=(a_i^k)_{1\le k\le d, 1\le i\le d-k+1}\in \mathbb{Z}^{d(d+1)/2}$
and ${\bf b}=(b_i^k)_{1\le k\le d, 1\le i\le d-k+1}\in
\mathbb{Z}^{d(d+1)/2}$,  if ${\bf c}=(c_i^k)_{1\le k\le
d, 1\le i\le d-k+1}\in \mathbb{Z}^{d(d+1)/2}$  is  such
that $M({\bf c})=M({\bf a})M({\bf b})$, then
\begin{equation}\label{c-eq-1}
c_i^k=a_i^k+(\sum
\limits_{j=1}^{k-1}a_i^{k-j}b_{i+k-j}^j)+b_i^k
\end{equation}
for $1\le k \le d$ and $1\le i \le d-k+1$.

Consider the following linear order relation on
$U_{d+1}(\mathbb{Z})$: $M({\bf a})<M({\bf b})$  if and only if
$$(a_d^1,a_{d-1}^1,\cdots,a_1^1;a_{d-1}^2,a_{d-2}^2,\cdots,a_1^2;a_{d-2}^3,\cdots,a_1^3;\cdots;a_2^{d-1},a_1^{d-1};a_1^d)$$ is lexicographically less than
$$(b_d^1,b_{d-1}^1,\cdots,b_1^1;b_{d-1}^2,b_{d-2}^2,\cdots,b_1^2;b_{d-2}^3,\cdots,b_1^3;\cdots;b_2^{d-1},b_1^{d-1};b_1^d),$$
where ${\bf a}=(a_i^k)_{1\le k\le d, 1\le i\le d-k+1}\in
\mathbb{Z}^{d(d+1)/2}$ and ${\bf b}=(b_i^k)_{1\le k\le d, 1\le i\le
d-k+1}\in \mathbb{Z}^{d(d+1)/2}$. This order relation is invariant
with respect to the translations of $U_{d+1}$ by \eqref{c-eq-1},
 so we can define an algebraic past $\Phi$ in
 $G$  as the subset of all elements of
$U_{d+1}(\mathbb{Z})$ which are less than the identity $I_{d+1}$,
 i.e.,
$$\Phi:=\{M({\bf a})\in U_{d+1}(\mathbb{Z}): M({\bf a})<I_{d+1}\}.$$
 As before, since  $G$ is nilpotent,
 it is amenable. Thus
$G=:U_{d+1}(\mathbb{Z})$ is a countable discrete
infinite amenable group with algebraic past $\Phi$.

 Consider the set of generators
for $U_{d+1}(\mathbb{Z})$ formed by the matrices $T_{i,j} = I_{d+1}
+ u_{i,j}, i < j, 1 \le i, j \le d+1$,  where each
$u_{i,j}$  is  a matrix unit  -- the matrix
whose $(k,p)$-entry equals $\delta_{ik}\delta_{jp}$ for all $1\le
k,p\le d+1$.  Let $f_n=T_{d,d+1}^n,
h_n=T_{d,d+1}^{-n}$, $n\in \mathbb{N}$,  and
\begin{equation}\label{s-upper-matrix}
S:=\{ M({\bf a}): (a_d^1)^k \ge a_i^k\ge 0 \text{ for }1\le k\le d, 1\le i\le d-k+1 \}.
\end{equation}
Then $S$ is a semigroup of $U_{d+1}(\mathbb{Z})$,
$S\subset \Phi^{-1}\cup \{I_{d+1}\}$ by \eqref{c-eq-1} and
\eqref{s-upper-matrix},
$$<S>=<\{T_{d,d+1}\}\cup \{T_{d,d+1}T_{i,j}: i < j, 1 \le i, j \le
d+1\}>= G,
$$ and $f_n\nearrow \infty$ w.r.t. $\Phi$ and $h_n\searrow \infty$ w.r.t. $\Phi$.
It is not hard to see that
\begin{align*}
&\hskip0.5cm \sharp\{s\in S:s<_\Phi f_n\}\\
&\le
\sharp\{(a_i^k)_{1\le k\le d, 1\le i\le d-k+1}\in
\mathbb{Z}^{d(d+1)/2}: n^k\ge a_i^k \ge 0 \text{ for }1\le k\le d, 1\le i\le d-k+1 \}\\
&<\infty,
\end{align*}
$f_n\Phi f_n^{-1}=\Phi$, $f_n\in S$ and $h_n\in S^{-1}$ for $n\ge
1$.  Summarizing up, $S$ is a $\Phi$-admissible
semigroup of the group  $G$, hence an application of
Theorem \ref{thm2}  yields the following result.

\begin{Theorem}\label{cor3} Let  $S$ be the semigroup defined in \eqref{s-upper-matrix}. Then
any positive entropy $U_{d+1}(\mathbb{Z})$-system has proper $S$-asymptotic
pairs.
More precisely, if $(X,U_{d+1}(\mathbb{Z}))$ is a $U_{d+1}(\mathbb{Z})$-system and
$\mu$ is a positive entropy ergodic $U_{d+1}(\mathbb{Z})$-invariant Borel
probability measure on $X$, then there exists $\delta>0$ such that
for $\mu$-a.e. $x\in X$, the stable set
$W_{S}(x,U_{d+1}(\mathbb{Z}))$ contains a
$(S^{-1},\delta)$-chaotic set which is a Cantor set.
\end{Theorem}

The paper is organized as follows. In Section 2, we
review some basic dynamical properties for $G$-systems. In Section
3, we establish  a Pinsker formula for a countable
discrete infinite amenable group $G$ with the algebraic past $\Phi$.
In Section 4,  we introduce Pinsker $\sigma$-algebra and investigate
some basic properties of it. The proof of the main results is given
in Section 5.

\section{Preliminary}
In this section,   we review some basic notions and
fundamental properties of $G$-systems. Let  $G$ be a countable discrete infinite amenable
group with unit element $e_G$. By a {\it $G$-system} $(X, G)$ we
mean that $X$ is a compact metric space and $\Gamma: G\times
X\rightarrow X, (g, x)\mapsto g x$ is a continuous mapping
 satisfying
\begin{enumerate}

\item $\Gamma (e_G, x)= x$ for each $x\in X$,

\item $\Gamma (g_1, \Gamma (g_2, x))= \Gamma (g_1 g_2, x)$ for each
$g_1, g_2\in G$ and $x\in X$.
\end{enumerate}
In the following, we fix a $G$-system $(X,G)$.  Then
under the induced continuous action on $X\times X$:
$g(x_1,x_2):=(gx_1,gx_2)$, $g\in G$, $(x_1,x_2)\in X\times X$,
$(X\times X,G)$ is also a $G$-system.

\subsection{F{\o}lner sequence and Entropy}
Let $F(G)$  be the set of
all finite non-empty subsets of $G$. If $(F_n)_{n\ge 1}$ is a F{\o}lner sequence of $G$, then
 \begin{equation}\label{folner-1}
 \lim_{n\rightarrow\infty}\frac{|F_n\triangle KF_n|}{|F_n|}=0 \text{ and }\lim_{n\rightarrow\infty} \frac{|\{ g\in F_n:Kg\subset F_n\}|}{|F_n|}=1
 \end{equation}
 for every $K\in F(G)$.

A {\it cover} of $X$ is a finite family of subsets of
$X$ whose union is $X$. A  {\it partition} of $X$ is a
cover of $X$ whose elements are pairwise disjoint. Denote by
$\mathcal{C}_X$ (resp. $\mathcal{C}_X^{o}$) the set of all covers
(resp. finite open covers) of $X$  and  by
$\mathcal{P}_X$ (resp. $\mathcal{P}_X^b$) the set of all partitions
of $X$ (resp. finite Borel partition).

Given two covers $\mathcal{U}$, $\mathcal{V}\in C_X$, $\mathcal{U}$
is said to be  {\it finer} than $\mathcal{V}$ (denoted
by $\mathcal{U} \succeq \mathcal{V}$) if each element of
$\mathcal{U}$ is contained in some element of $\mathcal{V}$. Let
$\mathcal{U}$ $\vee$ $\mathcal{V}$$= \{U \cap V : U \in \mathcal{U},
V \in \mathcal{V}\}$. Denote by $N(\mathcal{U})$ the number of sets
in a subcover of $\mathcal{U}$ of minimal cardinality.
The {\it entropy of $\mathcal{U}\in \mathcal{C}_X^{o}$ with respect
to $G$ is defined} by
\begin{equation*}
\label{TE}
h_{top}(G,\mathcal{U})=\lim\limits_{n\rightarrow+\infty}\frac{1}{|F_n|}
\log N(\bigvee_{g\in F_n}g^{-1}\mathcal{U}),
\end{equation*}
where $F_n$ is a F{\o}lner sequence in the group $G$. As is shown in
\cite[Theorem 6.1]{LW} (see also \cite{K,OW}) the limit exists and
is independent of  F{\o}lner sequences.
The {\it topological entropy} of $(X,G)$  is then
defined by
$$h_{top}(G)=h_{top}(G,X)=\sup_{\mathcal{ U}\in C_X^o}
h_{top}(G,\mathcal{U}).$$

Denote by $\mathcal{B}_X$ the collection of all Borel subsets of $X$
and $\mathcal{M}(X)$ the set of all Borel probability measures on
$X$. For $\mu\in \mathcal{M} (X)$, denote by $\text{supp} (\mu)$ the
{\it support} of $\mu$, i.e., the smallest closed subset $W\subseteq
X$ such that $\mu (W)= 1$. $\mu\in \mathcal{M} (X)$ is called {\it
$G$-invariant} if $g \mu= \mu$ for each $g\in G$,  and
called {\it ergodic} if  it is $G$-invariant and $\mu
(\bigcup_{g\in G} g A)= 0$ or $1$ for any $A\in \mathcal{B}_X$.
Denote by $\mathcal{M} (X, G)$ (resp. $\mathcal{M}^e(X,G)$) the set
of all $G$-invariant (resp. ergodic $G$-invariant) elements in
$\mathcal{M} (X)$. Note that the amenability of $G$ ensures that
 $\mathcal{M}^e (X, G)\ne \emptyset$ and both
$\mathcal{M}(X)$ and $\mathcal{M}(X, G)$ are convex compact metric
spaces when they are endowed with the weak$^*$-topology.

For $\mu\in \mathcal{M}(X)$, denote by $\mathcal{B}_X^\mu$ the completion of $\mathcal{B}_X$ under $\mu$,
and define
$$\mathcal{P}_X^\mu=\{ \alpha\in \mathcal{P}_X: \text{each element  in }\alpha \text{ belongs to }\mathcal{B}_X^\mu\}.$$
Given $\alpha\in \mathcal{P}^\mu_X$ and
 a sub-$\sigma$-algebra $\mathcal{A}$ of $\mathcal{B}_X^\mu$,
 define
\begin{equation*}
H_{\mu}
(\alpha | \mathcal{A})= \sum_{A\in \alpha} \int_X
 -\mathbb{E}_\mu (1_A| \mathcal{A}) \log \mathbb{E}_\mu (1_A| \mathcal{A}) d \mu,
\end{equation*}
where $\mathbb{E}_\mu (1_A| \mathcal{A})$ is the expectation of
$1_A$ with respect to  $\mathcal{A}$. One standard fact is that
$H_{\mu} (\alpha | \mathcal{A})$ increases with respect to $\alpha$
and decreases with respect to $\mathcal{A}$. Set $\mathcal{N}=
\{\emptyset, X\}$  and define
\begin{equation*}
H_\mu (\alpha)= H_\mu (\alpha| \mathcal{N})= \sum_{A\in \alpha}
-\mu(A) \log \mu(A).
\end{equation*}
Note that  any $\beta \in \mathcal{P}_X^\mu$ naturally
generates a sub-$\sigma$-algebra $\mathcal{F} (\beta)$ of
$\mathcal{B}_X^\mu$.  We then define
\begin{equation*}
H_{\mu}(\alpha|\beta)= H_\mu (\alpha| \mathcal{F} (\beta))=
H_{\mu}(\alpha\vee \beta)- H_{\mu}(\beta).
\end{equation*}

Given $\mu\in \mathcal{M}(X,G)$ and $\alpha\in \mathcal{P}_X^\mu$. The {\it measure-theoretic entropy} of $\mu$ relative to $\alpha$ is
defined by
\begin{equation*}
h_{\mu}(G,\alpha)=\lim\limits_{n\rightarrow+\infty}\frac{1}{|F_n|}
H_\mu(\bigvee_{g\in F_n}g^{-1}\alpha),
\end{equation*}
where $F_n$ is a Fl{\o}ner sequence in the group G.  As
shown in \cite[Theorem 6.1]{LW} (see also \cite{K,OW}),  the limit
exists and is independent of  F{\o}lner sequences. The
{\it measure-theoretic entropy} of $\mu$ is defined by
$$h_{\mu}(G)=h_{\mu}(G,X)=\sup_{\alpha\in \mathcal{P}^b_X}
h_{\mu}(G,\alpha).$$

A sub-$\sigma$-algebra $\mathcal{A}$ of $\mathcal{B}_X^\mu$ is said
to be {\it $G$-invariant} if $g\mathcal{A}=\mathcal{A}$ for any
$g\in G$. For the conditional entropy of $\alpha\in
\mathcal{P}_X^\mu$ with respect to a $G$-invariant sub
$\sigma$-algebra $\mathcal{A}$ of  $\mathcal{B}_X^\mu$ we
 define
$$h_\mu(G,\alpha|\mathcal{A}) = \lim \limits_{n\rightarrow+\infty}\frac{1}{|F_n|}
H_\mu(\bigvee_{g\in F_n}g^{-1}\alpha|\mathcal{A}),$$ where $F_n$ is
a F{\o}lner sequence in the group $G$. One can deduce
the existence of this limit and its independence of the sequence
$F_n$  from  \cite[Theorem 6.1]{LW} (see also
\cite{K,OW}  and  \cite{WZ} for an extension of the
approach of \cite{K} to the conditional case). The {\it conditional
entropy} of $\mu$ with respect to $\mathcal{A}$ is defined by
$$h_\mu(G|\mathcal{A})=\sup_{\alpha\in \mathcal{P}^b_X}
h_{\mu}(G,\alpha|\mathcal{A}).$$
It is well known that $h_{top}(G)=\sup_{\mu\in \mathcal{M}^e(X,G)}h_\mu(G)$ (see e.g. \cite{OP,ST}).

\subsection{Measurable partition and Relatively independent squares}

Let $\mu\in \mathcal{M}(X,G)$. Then $(X,\mathcal{B}_X^\mu,\mu,G)$ is
a Lebesgue system. If $\{ \alpha_i\}_{i\in I}$ is a countable family
of finite Borel partitions of $X$,  then the partition
$\alpha=\bigvee_{i\in I}\alpha_i$ is called a {\it measurable
partition}. The sets $A\in \mathcal{B}_X^\mu$, which are unions of
atoms of $\alpha$, form a sub-$\sigma$-algebra of
$\mathcal{B}_X^\mu$,  to be denoted  by
$\widehat{\alpha}$ or $\alpha$ if there is no ambiguity.
In fact, every sub-$\sigma$-algebra of
$\mathcal{B}_X^\mu$ coincides with a $\sigma$-algebra constructed in
 the way above (mod $\mu$). The following
result is well-known (see e.g., \cite[Theorem 14.26]{Gbook}).

\begin{Theorem} (Martingale Theorem) Let $(\mathcal{A}_n)_{n\ge 1}$ be a decreasing sequence (resp. an in creasing sequence) of sub-$\sigma$-algebras of $\mathcal{B}_X^\mu$
and let $\mathcal{A}= \bigcap_{n\ge 1} \mathcal{A}_n$ (resp.
$\mathcal{A}= \bigvee_{n\ge 1} \mathcal{A}_n$). Then for every $f
\in L^2(X,\mathcal{B}_X,\mu)$, $\mathbb{E}_\mu(f |
\mathcal{A}_n)\rightarrow \mathbb{E}_\mu(f | \mathcal{A})$ in
$L^2(\mu)$ and  also $\mu$-almost everywhere.
\end{Theorem}

Let $\mathcal{F}$ be a sub-$\sigma$-algebra
$\mathcal{B}_X^\mu$ and $\alpha$ be a measurable partition of
$X$ with $\widehat{\alpha}=\mathcal{F}$ (mod $\mu$). $\mu$ can be
disintegrated over $\mathcal{F}$ as $\mu=\int_X \mu_x d \mu(x)$
where $\mu_x\in \mathcal{M}(X)$  and $\mu_x(\alpha(x))=1$ for
$\mu$-a.e. $x\in X$. The disintegration is characterized by the
properties \eqref{meas1} and \eqref{meas3} below:
\begin{equation}\label{meas1}  \text{for every } f \in
L^1(X,\mathcal{B}_X,\mu),  f \in L^1(X,\mathcal{B}_X,\mu_x) \text{
for $\mu$-a.e. } x\in X,
\end{equation}
\hskip2cm and the map $x \mapsto \int_X  f(y)\,d\mu_x(y)$ is in
$L^1(X,\mathcal{F},\mu)$;

\begin{equation}
\label{meas3} \text{for every } f\in
L^1(X,\mathcal{B}_X,\mu),\mathbb{E}_{\mu}(f|\mathcal{F})(x)=\int_X
f\,d\mu_{x} \ \ \text{for $\mu$-a.e. } x\in X.
\end{equation}
Then for any $f \in L^1(X,\mathcal{B}_X,\mu)$,
\begin{equation*}
\int_X \left(\int_X f\,d\mu_x \right)\, d\mu(x)=\int_X f \,d\mu.
\end{equation*}
  Let $\mathcal{F}$ be a sub-$\sigma$-algebra of
$\mathcal{B}_X^\mu$ and $\mu=\int_X \mu_x d \mu(x)$ be the
disintegration of $\mu$ over $\mathcal{F}$. According to
\cite[Pages 231-232]{HF} and \cite[Pages 110-115]{Hbook}, the {\it conditional square} (or {\it conditional
product}) $\mu\times_\mathcal{F}\mu$ of $\mu$ relatively to
$\mathcal{F}$ is the  Borel probability measure
$\mu\times_\mathcal{F}\mu$ on $X\times X$ such that
\begin{equation}\label{ris}
\mu\times_\mathcal{F}\mu(A\times B)=\int_X \mathbb{E}_\mu
(1_A|\mathcal{F})(x)\mathbb{E}_\mu
(1_B|\mathcal{F})(x)d\mu(x)=\int_X \mu_x(A)\mu_x(B)d\mu(x)
\end{equation} for all $A,B\in \mathcal{B}_X$.
 It is clear that both projections of $\mu\times_\mathcal{F}\mu$ to $X$ are equal to
 $\mu$.

By standard arguments, for every pair of bounded Borel functions $f, h$ on $X$ one has
$$\int_{X\times X} f(x)h(y)d(\mu\times_\mathcal{F}\mu)(x,y)=\int_X \mathbb{E}_\mu(f|\mathcal{F})(x)\mathbb{E}_\mu(h|\mathcal{F})(x)d\mu(x).$$
By  \eqref{ris}, it is not hard to see that
\begin{equation}\label{eq-cdeq}
\mu \times_{g\mathcal{F}}\mu=g (\mu\times_{\mathcal{F}}\mu)
\end{equation}
for any $g\in G$ and  any sub-$\sigma$-algebra
$\mathcal{F}$ of $\mathcal{B}_X^\mu$, where $g(\mu\times_\mathcal{F}\mu)(E):=\mu\times_\mathcal{F}\mu(g^{-1}E)$
for any Borel subset $E$ of $X\times X$.

\section{Pinsker formula}
In this section, we will establish the  following
Pinsker formula for a countable discrete infinite amenable group $G$
with the algebraic past $\Phi$.

\begin{Theorem}\label{lm1} (Pinsker formula) Let $G$ be a countable discrete infinite amenable group
with algebraic past $\Phi$,  $(X,G)$ be a $G$-system,
$\mu\in \mathcal{M}(X,G)$, and $\mathcal{A}$ be a $G$-invariant
sub-$\sigma$-algebra of $\mathcal{B}_X^\mu$. Then for any
$\alpha,\beta\in \mathcal{P}^\mu_X$,
$$h_\mu(G,\alpha\vee\beta|\mathcal{A})=h_\mu(G,\beta|\mathcal{A})+H_\mu(\alpha|\beta_G\vee\alpha_\Phi\vee \mathcal{A}),$$
where $\beta_G=\bigvee_{g\in G}g\beta$ and
$\alpha_\Phi=\bigvee_{g\in \Phi}g\alpha$. In
particular, $h_\mu(G,\alpha|\mathcal{A})=H_\mu(\alpha|\alpha_\Phi
\vee \mathcal{A})$.
\end{Theorem}
\begin{proof} Let $(F_n)_{n\geq1}$ be a F{\o}lner sequence of $G$.
Denote $F_n^{-1}=\{g_{n,j}\}_{j=1}^{|F_n|}$ such that
$$g_{n,1}<_{\Phi} g_{n,2}\cdots<_{\Phi}g_{n,|F_n|}.$$
Then
\begin{align}\label{eq1}
h_\mu(G,\alpha\vee\beta|\mathcal{A})&=\lim_{n\rightarrow\infty}\frac{1}{|F_n|}H_\mu \big(\bigvee \limits_{g\in F_n}g^{-1}(\alpha\vee\beta)|\mathcal{A}\big)\nonumber\\
&=\lim_{n\rightarrow\infty}\frac{1}{|F_n|}\big(H_\mu \big(\bigvee
\limits_{g\in F_n}g^{-1}\beta|\mathcal{A})+
H_\mu(\bigvee \limits_{g\in F_n}g^{-1}\alpha|\bigvee \limits_{g\in F_n}g^{-1}\beta\vee \mathcal{A})\big)\nonumber\\
&=h_\mu(G,\beta|\mathcal{A})+\lim_{n\rightarrow\infty}\frac{1}{|F_n|}H_\mu(\bigvee
\limits_{g\in F_n}g^{-1}\alpha|\bigvee \limits_{g\in F_n}g^{-1}\beta \vee
\mathcal{A}).
\end{align}
 Note that
\begin{align*}\frac{1}{|F_n|}H_\mu(\bigvee
\limits_{g\in F_n}g^{-1}\alpha|\bigvee \limits_{g\in F_n}g^{-1}\beta \vee
\mathcal{A})&= \frac{1}{|F_n|}\sum_{i=1}^{|F_n|}H_\mu
\big(g_{n,i}\alpha|\bigvee_{j=1}^{i-1}g_{n,j}\alpha
\vee\bigvee_{g\in F_n}g^{-1}\beta \vee \mathcal{A}\big)\\
&=\frac{1}{|F_n|}\sum_{i=1}^{|F_n|}H_\mu
\big(\alpha|\bigvee_{j=1}^{i-1}g_{n,i}^{-1}
g_{n,j}\alpha\vee\bigvee_{g\in F_n}g_{n,i}^{-1}g^{-1}\beta \vee \mathcal{A}\big)\\
&\geq \frac{1}{|F_n|}\sum_{i=1}^{|F_n|}H_\mu(\alpha|\alpha_\Phi\vee\beta_G \vee \mathcal{A})\\
&=H_\mu(\alpha|\alpha_\Phi\vee\beta_G \vee \mathcal{A}).
\end{align*}
Combining this with \eqref{eq1}, we have
\begin{align}\label{eq3}
h_\mu(G,\alpha\vee\beta|\mathcal{A})\geq
h_\mu(G,\beta|\mathcal{A})+H_\mu(\alpha|\alpha_\Phi\vee\beta_G \vee
\mathcal{A}).
\end{align}

Since $(F_n)_{n\ge 1}$ is  a F{\o}lner sequence of $G$, it follows
from \eqref{folner-1} that, for  given $\epsilon>0$ and
$M,L\in F(G)$ with $M\subset\Phi$, $L\subset G$, there exists a
natural number $N$, such that, whenever $n\ge N$,
    $$|\{g\in F_n|M^{-1}g\subset F_n\}|\geq(1-\epsilon)|F_n|\text{ and }|\{g\in F_n|L^{-1}g\subset F_n\}|\geq(1-\epsilon)|F_n|.$$
Thus for each $n\ge N$, there are at least $(1-2\epsilon)|F_n|$
indices $i$  in $\{1,2,\cdots,|F_n|\}$ satisfying $M^{-1}\subseteq
F_n g_{n,i}$ and $L^{-1}\subseteq F_n g_{n,i}$.  Since
$M^{-1}\subseteq \Phi^{-1}$ and $\Phi^{-1}\cap F_n g_{n,i}=\{
g_{n,j}^{-1} g_{n,i}:1\le j\le i-1\}$,  we have that,
for each $n\ge N$, there are at least $(1-2\epsilon)|F_n|$ indices
$i$ in $\{1,2,\cdots,|F_n|\}$ satisfying $M^{-1}\subseteq \bigcup
\limits_{j=1}^{i-1} \{g_{n,j}^{-1}g_{n,i}\}$ and $L^{-1}\subseteq
F_n g_{n,i}$,  or equivalently,  $M\subseteq \bigcup
\limits_{j=1}^{i-1} \{g_{n,i}^{-1}g_{n,j}\}$ and $L\subseteq
g_{n,i}^{-1}F_n^{-1}$.

For every $n\geq N$, we have
\begin{align*}
&\hskip0.5cm \frac{1}{|F_n|}H_\mu(\bigvee \limits_{g\in
F_n}g^{-1}\alpha|\bigvee
\limits_{g\in F_n}g^{-1}\beta \vee \mathcal{A})\\
&=\frac{1}{|F_n|}\sum_{i=1}^{|F_n|}H_\mu(\alpha|\bigvee_{j=1}^{i-1}g_{n,i}^{-1}
g_{n,j}\alpha\vee\bigvee_{g\in F_n}g_{n,i}^{-1}g^{-1}\beta \vee \mathcal{A})\\
&\leq \frac{1}{|F_n|}\big((1-2\epsilon)|F_n|H_\mu(\alpha|\bigvee_{g\in M}g\alpha\vee\bigvee_{g\in L}g\beta \vee \mathcal{A})
+2\epsilon |F_n|H_\mu(\alpha|\mathcal{A})\big)\\
&=(1-2\epsilon)H_\mu(\alpha|\bigvee_{g\in M}g\alpha\vee\bigvee_{g\in
L}g\beta \vee \mathcal{A})+2\epsilon H_\mu(\alpha|\mathcal{A}).
\end{align*}
 By letting $N\rightarrow \infty$ in the above inequality and \eqref{eq1}, we have
$$h_\mu(G,\alpha\vee \beta|\mathcal{A})\le h_\mu(G,\beta|\mathcal{A})+(1-2\epsilon)H_\mu(\alpha|\bigvee_{g\in
M}g\alpha\vee\bigvee_{g\in L}g\beta\vee \mathcal{A})+2\epsilon
H_\mu(\alpha|\mathcal{A}).$$  For fixed $M,L\in F(G)$
with $M\subset\Phi$, $L\subset G$,   letting $\epsilon
\searrow 0$ yields that
$$h_\mu(G,\alpha\vee\beta|\mathcal{A})\leq
h_\mu(G,\beta|\mathcal{A})+H_\mu(\alpha|\bigvee_{g\in
M}g\alpha\vee\bigvee_{g\in L}g\beta \vee \mathcal{A}).$$  Letting
$M\to \Phi$ and $L\to G$ in the above further yields
that
$$h_\mu(G,\alpha\vee\beta|\mathcal{A})\leq
h_\mu(G,\beta|\mathcal{A})+H_\mu(\alpha|\alpha_\Phi\vee\beta_G \vee
\mathcal{A}).$$ Combing this with \eqref{eq3}, we have
$h_\mu(G,\alpha\vee\beta|\mathcal{A})=
h_\mu(G,\beta|\mathcal{A})+H_\mu(\alpha|\alpha_\Phi\vee\beta_G \vee
\mathcal{A})$. This  completes  the proof.
\end{proof}

Let $(X,G)$ be a $G$-system and $\mu\in
\mathcal{M}(X,G)$ be as in the theorem above, it is already known
that $h_\mu(G,\alpha)=H_\mu(\alpha|\alpha_\Phi)$,  $\alpha \in
\mathcal{P}^\mu_X$ (see \cite{P,AV}).   The Pinsker formula is also
shown in \cite{C} for the case $G=\mathbb{Z}^d$ and in \cite[Theorem
4.1]{GS} when $G$ is a  finitely-generated, nilpotent group.
Moreover, an entropy addition formula for amenable groups is
obtained in \cite[Lemma 4.1]{WZ} and \cite[Lemma 1.1]{GTW}.

As an application of Theorem \ref{lm1}, we have the following result which will be used in the proof of our main results.
\begin{Proposition} \label{prop3.1} Let $G$ be a countable infinite amenable group with the algebraic past
$\Phi$ and $f_n\nearrow \infty$ w.r.t $\Phi$ with $f_n\Phi
f_n^{-1}=\Phi$ for each $n\ge 1$.  Also let $(X,G)$
 be a $G$-system, $\mu\in \mathcal{M}(X,G)$,
$\mathcal{A}$  be  a $G$-invariant sub-$\sigma$-algebra
of $\mathcal{B}_X^\mu$,  and $\alpha,\beta,\gamma\in
\mathcal{P}_X^\mu$ with $\alpha\preceq\beta$. Then
$$\lim \limits_{n\rightarrow\infty} H_\mu(\alpha|\beta_\Phi\vee (f_n^{-1}\gamma)_\Phi\vee \mathcal{A})=H_\mu(\alpha|\beta_\Phi\vee \mathcal{A}).$$
\end{Proposition}
\begin{proof}
 It follows from Theorem \ref{lm1} and  the fact $f_n\Phi f_n^{-1}=\Phi$ that
\begin{align}\label{eq-17}
&\hskip0.5cm H_\mu(\beta|\beta_\Phi\vee (f_n^{-1}\gamma)_\Phi\vee
\mathcal{A})\nonumber\\
& =H_\mu(\beta\vee f_n^{-1}\gamma|\beta_\Phi\vee
(f_n^{-1}\gamma)_\Phi\vee \mathcal{A})
-H_\mu(f_n^{-1}\gamma|\beta\vee\beta_\Phi\vee (f_n^{-1}\gamma)_\Phi\vee \mathcal{A})\nonumber\\
&=h_\mu(G,\beta\vee f_n^{-1}\gamma|\mathcal{A})-H_\mu(f_n^{-1}\gamma|\beta\vee\beta_\Phi\vee (f_n^{-1}\gamma)_\Phi\vee \mathcal{A})\nonumber\\
&=H_\mu(\beta|\beta_\Phi\vee \mathcal{A})+H_\mu(f_n^{-1}\gamma|\beta_G\vee (f_n^{-1}\gamma)_\Phi\vee \mathcal{A})
-H_\mu(f_n^{-1}\gamma|\beta\vee\beta_\Phi\vee (f_n^{-1}\gamma)_\Phi\vee \mathcal{A})\nonumber\\
&=H_\mu(\beta|\beta_\Phi\vee \mathcal{A})+H_\mu(\gamma|\beta_G\vee
f_n(f_n^{-1}\gamma)_{\Phi}\vee \mathcal{A})-H_\mu(\gamma|f_n\beta\vee f_n\beta_\Phi\vee
f_n(f_n^{-1}\gamma)_{\Phi}\vee \mathcal{A})\nonumber \\
&=H_\mu(\beta|\beta_\Phi\vee \mathcal{A})+H_\mu(\gamma|\beta_G\vee
\gamma_{\Phi }\vee \mathcal{A})-H_\mu(\gamma|f_n\beta\vee
f_n\beta_\Phi\vee \gamma_{\Phi }\vee \mathcal{A}).
\end{align}
Since  $f_n\nearrow \infty$ w.r.t $\Phi$, $f_n\beta
\vee f_n\beta_\Phi\nearrow \beta_G$ as $n\rightarrow \infty$.
Moreover, $\lim_{n\rightarrow\infty}H_\mu(\gamma|f_n\beta\vee
f_n\beta_\Phi\vee \gamma_\Phi\vee
\mathcal{A})=H_\mu(\gamma|\beta_G\vee \gamma_\Phi\vee \mathcal{A})$
by the Martingale Theorem. Thus,  by
letting $n\rightarrow\infty$ in \eqref{eq-17}, we have
\begin{align*}
&\hskip0.5cm\lim_{n\rightarrow\infty}H_\mu(\beta|\beta_\Phi\vee
(f_n^{-1}\gamma)_\Phi\vee \mathcal{A})\\
&=\lim_{n\rightarrow\infty}H_\mu(\beta|\beta_\Phi\vee \mathcal{A})
+H_\mu(\gamma|\beta_G\vee \gamma_\Phi\vee \mathcal{A})
-H_\mu(\gamma|f_n\beta\vee f_n\beta_\Phi\vee \gamma_\Phi\vee \mathcal{A})\nonumber\\
&=H_\mu(\beta|\beta_\Phi\vee \mathcal{A}).
\end{align*}
Combing this with the identity
\begin{align*}
 H_\mu(\alpha|\beta_\Phi\vee (f_n^{-1}\gamma)_\Phi\vee \mathcal{A})=H_\mu(\beta|\beta_\Phi\vee (f_n^{-1}\gamma)_\Phi\vee \mathcal{A})
 -H_\mu(\beta|\alpha\vee\beta_\Phi\vee (f_n^{-1}\gamma)_\Phi\vee \mathcal{A}),
\end{align*} we have
\begin{align*}
\liminf_{n\rightarrow\infty}H_\mu(\alpha|\beta_\Phi\vee
(f_n^{-1}\gamma)_\Phi\vee \mathcal{A})&\geq
\liminf_{n\rightarrow\infty} H_\mu(\beta|\beta_\Phi\vee
(f_n^{-1}\gamma)_\Phi\vee \mathcal{A})
-H_\mu(\beta|\alpha\vee\beta_\Phi\vee \mathcal{A})\\
&= H_\mu(\beta|\beta_\Phi\vee \mathcal{A})-H_\mu(\beta|\alpha\vee\beta_\Phi\vee \mathcal{A})\\
&=H_\mu(\alpha|\beta_\Phi\vee \mathcal{A}).
\end{align*}
Since
$\limsup_{n\rightarrow\infty}H_\mu(\alpha|\beta_\Phi\vee
(f_n^{-1}\gamma)_\Phi\vee \mathcal{A})\leq H_\mu(\alpha|\beta_\Phi\vee
\mathcal{A})$,
$$\lim_{n\rightarrow\infty}H_\mu(\alpha|\beta_\Phi\vee
(f_n^{-1}\gamma)_\Phi\vee \mathcal{A})= H_\mu(\alpha|\beta_\Phi\vee
\mathcal{A}).$$ This  completes the proof.
\end{proof}

\section{Pinsker $\sigma$-algebra}
In this section, we will introduce Pinsker $\sigma$-algebra and
investigate some of its basic properties
to be used in the proof of our main results.  Let $G$ be a
countable infinite amenable group and $(X,G)$ be a $G$-system.  For
$\mu\in \mathcal{M}(X,G)$ and  a $G$-invariant sub-$\sigma$-algebra
$\mathcal{A}$ of $\mathcal{B}_X^\mu$, denote
$$P_\mu(G|\mathcal{A})=\{ A\in \mathcal{B}_X^\mu:
h_\mu(G,\{A,X\setminus A\}|\mathcal{A})=0 \}.$$  It
follows from Theorem \ref{lm1} or \cite[Lemma 1.1]{GTW} that
$P_\mu(G|\mathcal{A})$  must be a $G$-invariant
sub-$\sigma$-algebra of $\mathcal{B}_X^\mu$ containing
$\mathcal{A}$. We call this $\sigma$-algebra the {\it Pinsker
$\sigma$-algebra} of the system $(X,\mathcal{B}_X^\mu,\mu,G)$
relative to $\mathcal{A}$. We simply refer to the
Pinsker $\sigma$-algebra of  $(X,\mathcal{B}_X^\mu,\mu,G)$ relative
to the trivial algebra as the {\it Pinsker $\sigma$-algebra},
 denoted  by $P_\mu(G)$.

Descriptions of Pinsker algebras are already given in
\cite{C} for $\mathbb{Z}^d$-actions  and in \cite{GS} for actions of
a finitely-generated nilpotent group.    Also,  some new properties
of Pinsker algebras related to disjointness and quasi-factors are
discovered in \cite{GTW}, followed by  a relativized version given
in  \cite{D} in which  a notion of entropy for cocycles is
introduced.

The following result is well-known. For  the sake of
completeness, we  provide  a proof  below.
\begin{Lemma}\label{lem-pinsker} Let  $G$ be a countable infinite amenable group,  $(X,G)$ be a $G$-system,  and $\mu\in
\mathcal{M}(X,G)$.  Then the Pinsker $\sigma$-algebra of the system
$(X,\mathcal{B}_X^\mu,\mu,G)$ relative to $P_\mu(G)$
agrees with $P_\mu(G)$, i.e.,
$$P_\mu(G|P_\mu(G))=P_\mu(G).$$
\end{Lemma}
\begin{proof} It is clear that $P_\mu(G|P_\mu(G))\supseteq P_\mu(G)$. Now we take an increasing
sequence of finite Borel partitions $(\beta_n)_{n\ge 1}$ such that
$\widehat{\beta_n}\nearrow P_\mu(G)$  as $n\rightarrow
+\infty$.  For a given $\alpha\in \mathcal{P}_X^\mu$, we
have by  \cite[Equation ($^*$) \& Lemma 1.1(1)]{GTW} and the Martingale Theorem that
\begin{align*}
h_\mu(G,\alpha)&\ge h_\mu(G,\alpha|P_\mu(G))=\lim_{n\rightarrow +\infty} h_\mu(G,\alpha| (\beta_n)_G)\\
&=\lim_{n\rightarrow +\infty} \big(h_\mu(G,\alpha\vee \beta_n)-h_\mu(G,\beta_n)\big)\\
&=\lim_{n\rightarrow +\infty} h_\mu(G,\alpha\vee \beta_n)\\
&\ge h_\mu(G,\alpha).
\end{align*}
Thus $h_\mu(G,\alpha|P_\mu(G))=h_\mu(G,\alpha)$. This implies $P_\mu(G|P_\mu(G))=P_\mu(G)$.
\end{proof}

 With Lemma \ref{lem-pinsker}, the following result
follows from Theorem 0.4(iii) in \cite{D} (see also Theorem 4 in \cite{GTW} for free action).
\begin{Lemma} \label{lem-GTW} Let  $G$ be a countable infinite amenable group, $(X,G)$ be a $G$-system, and $\mu\in \mathcal{M}^e(X,G)$.
If $\lambda=\mu\times_{P_\mu(G)}\mu$, and $\pi: X\times X\rightarrow
X$ is the  canonical projection to the first  factor,
then $P_\lambda(G|\pi^{-1}(P_\mu(G)))=\pi^{-1}(P_\mu(G))$ (mod
$\lambda$).
\end{Lemma}

\begin{Lemma}\label{bl} Let  $G$ be a countable infinite amenable group,
$(X,G)$ be a $G$-system, and $\mu\in \mathcal{M}^e(X,G)$.
Denote $\lambda=\mu\times_{P_\mu(G)}\mu$ and
$\Delta_X=\{(x,x):x\in X\}$.  Then the following holds:
\begin{itemize}
\item[{\rm 1)}] $\lambda\in \mathcal{M}^e(X\times X,G)$;

\item[{\rm 2)}] If $h_\mu(G)>0$, then $\lambda(\Delta_X)=0$.
\end{itemize}
\end{Lemma}
\begin{proof} 1) Let $A$ be a Borel subset of $X\times X$ and $B:=\bigcup_{g\in G}gA$. Then $gB=B$ for any $g\in G$ and thus
$h_\lambda(G, \{B,X\times X\setminus B\})=0$, i.e., $B\in
P_\lambda(G)$. Let $\pi: X\times X\rightarrow X$ be the  canonical
projection to the first  factor. By Lemma
\ref{lem-GTW}, we have
$P_\lambda(G|\pi^{-1}(P_\mu(G)))=\pi^{-1}(P_\mu(G))$ (mod
$\lambda$). Thus $P_\lambda(G)\subseteq
P_\lambda(G|\pi^{-1}(P_\mu(G)))=\pi^{-1}(P_\mu(G))$ (mod $\lambda$).
Let  $C\in P_\mu(G)$ be such that $B=\pi^{-1}(C)$ (mod
$\lambda$).  Using the fact that $gB=B$, $g\in G$, we
have $\lambda(g\pi^{-1}(C)\Delta \pi^{-1}C)=0$, $g\in G$. Moreover,
$$\mu(gC\Delta C)=\lambda(\pi^{-1}(gC\Delta C))=\lambda(g\pi^{-1}(C)\Delta \pi^{-1}C)=0,\;\;\, g\in G.
$$
Since $\mu$ is ergodic, $\mu(C)=0$ or $1$. Thus
$\lambda(B)=\mu(C)=0$ or $1$,  i.e.,  $\lambda$ is
ergodic.

2) Assume that $h_\mu(G)>0$. If $\lambda(\Delta_X)>0$, then
$\lambda(\Delta_X)=1$ since $\lambda$ is ergodic and $\Delta_X$ is
$G$-invariant. Now for any $A\in \mathcal{B}_X$,  we have
$$0=\lambda(A\times (X\setminus A))=\int_X \mathbb{E}_\mu(1_A|P_\mu(G))(x) \mathbb{E}_\mu(1_{X\setminus A}|P_\mu(G))(x) d\mu(x).$$
Thus the product of the two conditional expectations is equal to $0$ a.e. As the sum
of these two functions is equal to $1$, each of them is equal to $0$ or $1$ a.e. It follows
that $\mathbb{E}_\mu(1_A |P_\mu(G)) = 1_A$ a.e., and $A\in P_\mu(G)$ (mod $\mu$). Thus the $\sigma$-algebras $\mathcal{B}_X$ and $P_\mu(G)$
 are equal up to null sets. This implies $h_\mu(G)=0$,  a contradiction.
\end{proof}

The next lemma establishes a connection among asymptotic pairs,
Pinsker $\sigma$-algebra,  and entropy.
\begin{Lemma}\label{lem} Let $G$ be a countable discrete infinite amenable group with the algebraic past $\Phi$,  $f_n\nearrow \infty$ w.r.t. $\Phi$ with $f_n\Phi f_n^{-1}=\Phi$ for each $n\ge 1$, and
$S$ be a infinite subset of $G$ such that $\sharp\{s\in S:s<_\Phi
f_n\}<\infty$ for each $n\ge 1$.  Also let  $(X,G)$ be
a $G$-system and $\mu\in \mathcal{M}(X,G)$.   Then
there exists a measurable partition $\mathcal{P}$ of $(X, G,\mu)$
such that  the following holds:
\begin{enumerate}
\item $\overline{\mathcal{P}_\Phi(x)}\subseteq W_{Sh}(x,G)$ for any $x\in X$ and $h\in \Phi^{-1}\cup \{e_G\}$, where
$\mathcal{P}_\Phi=\bigvee_{g\in \Phi}g\mathcal{P}$ and
$\mathcal{P}_\Phi(x)$ is the atom of $\mathcal{P}_\Phi$ containing
x.

\item $\bigcap_{h\in \Phi} h(\widehat{\mathcal{P}_\Phi}\vee P_\mu(G))=P_\mu(G)$.

\item If, in addition, $h_\mu(G)>0$, then $\widehat{\mathcal{P}_\Phi}\neq \mathcal{B}_X^\mu \, (mod \,
\mu)$.
\end{enumerate}
\end{Lemma}
\begin{proof} Denote the metric on $X$ by $d$. Let $\{\alpha_n\}_{n\ge 1}$ be an increasing
sequence of finite Borel partitions of $X$ such that
$\lim_{n\rightarrow \infty}\text{diam}(\alpha_n)=0$.
Applying  Proposition \ref{prop3.1} inductively, we can find a
sequence $k_1,k_2,\cdots$ such that $k_i\nearrow \infty$ and for
each $q\ge 2$,
\begin{align}\label{bound}
H_\mu(P_k|(P_{q-1})_\Phi\vee P_\mu(G))-H_\mu(P_k|(P_q)_\Phi\vee
P_\mu(G))<\frac{1}{k2^{q-k}},\ k=1,2,\cdots,q-1,
\end{align}
 where $P_j=\bigvee_{i=1}^j f_{k_i}^{-1}\alpha_i$.
 We want to show that the measurable partition
 $\mathcal{P}=:\bigvee_{i=1}^\infty P_i$ satisfies the properties (1)-(3) above.

(1) Let $x\in X$, $y\in \overline{\mathcal{P}_\Phi(x)}$ and $h\in \Phi^{-1}\cup \{e_G\}$.
For a given $\epsilon>0$, since $\lim_{n\rightarrow
\infty}\text{diam}(\alpha_n)=0$, there exists $i\in \mathbb{N}$ such
that $\text{diam}(\alpha_i)\le \epsilon$. Now for any $s\in S$ with
$s>_\Phi f_{k_i}$, one has $s^{-1}f_{k_i}\in \Phi$ and
hence $(sh)^{-1}f_{k_i}\in \Phi$ and $y\in \overline{\mathcal{P}_\Phi(x)}\subseteq \overline{\big(
(sh)^{-1}f_{k_i}(f_{k_i}^{-1}\alpha_i)\big)(x)}=\overline{\big(
(sh)^{-1}\alpha_i\big)(x)}$. Thus $shy\in \overline{\alpha_i(shx)}$.
 It follows that $d(shx,shy)\le \text{diam}(\alpha_i)\le
\epsilon$ for any $s\in S$ with $s>_\Phi f_{k_i}$. Combing this with
the fact that
 $\sharp\{s\in S:s<_\Phi f_{k_i}\}<\infty$, we know that there are only
 finitely many $s\in S$ with $d(shx,shy)>\epsilon$. Since
$\epsilon$ is arbitrary, $(x,y)$ is an $Sh$-asymptotic pair. Thus $y\in
W_{Sh}(x,G)$.

\medskip
(2) We note by the $G$-invariance of $P_\mu(G)$ that
$gP_\mu(G)=P_\mu(G)$ for all $g\in G$.  By
also noting that $\Phi\Phi\subset \Phi$,  we have
\begin{equation}\label{P1}
\widehat{\mathcal{P}_\Phi}\vee P_\mu(G)\supset \bigcup_{g\in
\Phi}g(\widehat{\mathcal{P}_\Phi}\vee P_\mu(G)) \supset\bigcap_{g\in
\Phi}g(\widehat{\mathcal{P}_\Phi}\vee P_\mu(G))\supseteq P_\mu(G).
\end{equation}
For any $A\in \bigcap_{h\in
\Phi}h(\widehat{\mathcal{P}_\Phi}\vee P_\mu(G))$,  we
let $\xi=\{ A,X\setminus A\}$. Given $g\in G$,  if $g\in \Phi\cup
\{e_G\}$, then $g\widehat{\xi}\subseteq
g(\widehat{\mathcal{P}_\Phi}\vee P_\mu(G))\subseteq
\widehat{\mathcal{P}_\Phi}\vee P_\mu(G)$. If $g\in \Phi^{-1}$, then
$g^{-1}\in \Phi$ and
$$g\widehat{\xi}\subseteq
gg^{-1}(\widehat{\mathcal{P}_\Phi}\vee P_\mu(G))=
\widehat{\mathcal{P}_\Phi}\vee P_\mu(G).$$ Hence
$g\widehat{\xi}\subseteq \widehat{\mathcal{P}_\Phi}\vee P_\mu(G)$
for any $g\in G$. Thus  $\widehat{\xi_G}=\bigvee_{g\in G}
g\widehat{\xi}\subseteq \widehat{\mathcal{P}_\Phi}\vee P_\mu(G)$.
Moreover,  by Theorem \ref{lm1} and \eqref{bound}, we
have
   \begin{align*}
   &\hskip0.5cm H_\mu(\xi|\xi_\Phi\vee P_\mu(G))\\
   &=H_\mu(\xi\vee P_i|\xi_\Phi\vee(P_{i})_\Phi\vee P_\mu(G))-H_\mu(P_i|(P_i)_\Phi\vee\xi_{G}\vee P_\mu(G))\\
   &\leq H_\mu(\xi| (P_{i})_\Phi\vee P_\mu(G))+\big( H_\mu( P_i|(P_{i})_\Phi\vee P_\mu(G))-H_\mu(P_i|\mathcal{P}_\Phi\vee P_\mu(G)\big)\\
   &=H_\mu(\xi| (P_{i})_\Phi\vee P_\mu(G))+\sum_{j=i}^\infty\big( H_\mu(
   P_i|(P_j)_\Phi\vee P_\mu(G))-H_\mu(P_i|(P_{j+1})_\Phi\vee P_\mu(G))\big)\\
   &\leq
   H_\mu(\xi|(P_{i})_\Phi\vee P_\mu(G))+\sum_{j=i}^{\infty}\frac{1}{i2^{j+1-i}} =H_\mu(\xi|(P_{i})_\Phi\vee P_\mu(G))+\frac{1}{i}.
   \end{align*}
Letting  $i\to \infty$ in the above yields that $$h_\mu(G,\xi)=H_\mu(\xi|\xi_\Phi\vee P_\mu(G))\leq
H_\mu(\xi|\mathcal{P}_\Phi\vee P_\mu(G))=0.$$ Hence $A\in P_\mu(G)$.
Since $A$ is arbitrary,
$$\bigcap_{h\in\Phi}h(\widehat{\mathcal{P}_\Phi}\vee P_\mu(G))\subset P_\mu(G).$$
 This, together with \eqref{P1}, proves (2).

(3) Suppose for contradiction that  $\widehat{\mathcal
{P}_\Phi}=\mathcal {B}_X^\mu \, (mod \ \mu)$.  Then
   $h\widehat{\mathcal{P}_\Phi} =\mathcal{B}_X^\mu \,(mod\, \mu)$ for any $h\in
   G$. Hence $\bigcap_{h\in\Phi}h(\widehat{P_\Phi
   }\vee P_\mu(G))=\mathcal{B}_X^\mu$. It follows from (2) that $P_\mu(G)=\mathcal{B}_X^\mu$,  and consequently, $h_\mu(G)=0$, a contradiction.
\end{proof}

\section{Proof of main results}

We first prove  Theorem \ref{thm1}.
\begin{proof}[Proof of Theorem \ref{thm1}] Let $(X,G)$ be a $G$-system and $\mu\in \mathcal{M}^e(X,G)$ with
$h_\mu(G)>0$. Then by Lemma \ref{lem} there exists a
measurable partition $\mathcal{P}$ of $(X,\mathcal{B}_X^\mu, G, \mu)$
such that
\begin{itemize}
 \item $\mathcal{P}_\Phi(x) \subseteq \bigcap_{h\in \Phi^{-1}\cup\{e_G\}} W_{Sh}(x,G)$ for any $x\in X$,
 and

 \item $\widehat{\mathcal{P}_\Phi}\neq \mathcal{B}_X^\mu \, (mod \, \mu)$.
\end{itemize}
  Let
$$E=\{(x,y)\in X\times X: (x,y) \text{ is }\text{an $Sh$-asymptotic pair for all }h\in \Phi^{-1}\cup\{e_G\}\}.$$
Then $E$ is a Borel subset of $X\times X$ with $gE\subset E$ for all $g\in \Phi^{-1}\cup\{e_G\}$. Let
$$J=\pi(E\setminus \Delta_X),$$
where $\pi:X\times X\rightarrow X$  denotes the
projection onto the first factor and $\Delta_X:=\{(x,x):x\in X\}$,
i.e., $J=\{x\in X: \bigcap_{h\in \Phi^{-1}\cup\{e_G\}}W_{Sh}(x,G)\setminus \{x\}\neq \emptyset\}$.  By a
theorem of Lusin that analytic sets are universally measurable (see e.g., \cite[Theorem 21.10]{CDS-K}), we have $J\in \mathcal{B}_X^\mu$. Note that $gJ\subseteq J$ for all $g\in \Phi^{-1}\cup\{e_G\}$, so
$\mu(gJ\Delta J)=0$ for $g\in \Phi^{-1}\cup\{e_G\}$ and hence $\mu(tJ\Delta J)=0$ for all $t\in G$ as $G=\Phi\cup \Phi^{-1}\cup\{e_G\}$.  By ergodicity of $\mu$, $\mu(J) = 0$ or
$1$.

If  $\mu(J)$ = 0, then $\bigcap_{h\in \Phi^{-1}\cup\{e_G\}}W_{Sh}(x,G) = \{x\}$ for $\mu$-a.e. $x\in X$.
Thus $\mathcal{P}_\Phi(x) = \{x\}$  for $\mu$-a.e. $x\in X$. This
implies that $\widehat{\mathcal{P}_\Phi}=\mathcal{B}_X^\mu
\, (mod \, \mu)$,  a
contradiction to  $\widehat{\mathcal{P}_\Phi}\neq \mathcal{B}_X^\mu \, (mod \, \mu)$. Hence $\mu(J)$ =1, which implies that $W_S(x,G)\setminus \{x\}\neq
\emptyset$ for $\mu$-a.e. $x\in X$.
\end{proof}

The proof of Theorem \ref{thm2}  will
need the following result due to Mycielski (see e.g. \cite[Lemma 5.9 and Theorem
5.10]{Ak} and \cite[Theorem 6.32]{AAG}).

\begin{Lemma}(Mycielski) \label{Myc} Let $Y$ be a perfect compact metric space and $C$ be a symmetric dense
$G_\delta$ subset of $Y\times Y$. Then there exists a dense subset
$K\subseteq Y$ which is  a union of countably many Cantor sets such
that $K\times K\subseteq C\cup \Delta_Y$, where $\Delta_Y=\{
(y,y):y\in Y\}$.
\end{Lemma}

\begin{proof}[Proof of Theorem \ref{thm2}] Let $(X,G)$ be a
$G$-system and $\mu \in \mathcal{M}^e(X,G)$ with $h_\mu(G)>0$, where
$X$ is endowed  with the metric $d$. Denote
$\lambda=\mu\times_{P_\mu(G)}\mu$. Take $z_0\in supp(\mu)$. For any open neighborhood $A$ of $(z_0,z_0)$ in $X\times X$, there exists
an open neighborhood $B$ of $x$ in $X$ such that $B\times B\subseteq A$. As  $z_0\in supp(\mu)$, $\mu(B)>0$.
Moreover
\begin{align*}
\lambda(A)&\ge \lambda(B\times B)=\int_X \mathbb{E}_\mu
(1_B|P_\mu(G))(x) \mathbb{E}_\mu
(1_B|P_\mu(G))(x)d\mu(x)\\
&\ge \big(\int_X \mathbb{E}_\mu
(1_B|P_\mu(G))(x)d\mu(x)\big)^2=\mu(B)^2\\
&>0,
\end{align*}
this implies $(z_0,z_0)\in supp(\lambda)$.

By Lemma \ref{bl},
$\lambda\in \mathcal{M}^e(X\times X,G)$ and $supp(\lambda)\varsubsetneq\Delta_X$. Take $(x_0,y_0)\in supp(\lambda)$ such that
$d(x_0,y_0)>0$. For each $i\in\mathbb{N}$, let $U_i$
(resp. $V_i$) be the open ball centered
at $(x_0,y_0)$ (resp. $(z_0,z_0)$) and
of radius $\frac{1}{i}$. Since $S$ is a $\Phi$-admissible
semigroup and $G$ is an infinite torsion-free group, $S$ is an
infinite set. For  each $g\in G$, we define
$$U_{i,g}=\bigcup_{t\in S\setminus \{e_G\}} t g(U_i)\qquad\quad
 {\rm (}resp. \;V_{i,g}=\bigcup_{t\in S\setminus
\{e_G\}}t g(V_i){\rm)}.$$ Let $\delta:=\frac{1}{2}d(x_0,y_0)$ and
$$W=\bigcap_{i\geq 1}\Big( (\bigcap_{g\in
S}U_{i,g})\cap(\bigcap_{g\in S}V_{i,g}) \Big).$$

\medskip

\noindent{\bf Claim 1.}  Any $(x,y)\in W$ is a
$(S^{-1},\delta)$-Li-Yorke pair.

\medskip
Let $(x,y)\in W$ and set $E=\{ s\in
S:d(s^{-1}x,s^{-1}y)>\delta\}$. We  first show that $E$
is an infinite set. Take $r\in \mathbb{N}$ such that $d(u,v)>\delta$
for any $(u,v)\in U_{r}$. Take $g\in S$. Then $(x,y)\in U_{r,g}=\bigcup_{t\in S\setminus \{e_G\}} t g(U_i)$. Thus there exists $t\in S\setminus \{e_G\}$ such that $((tg)^{-1}x,(tg)^{-1}y)\in U_r$, and consequently, $tg\in E$. Hence $E$ is nonempty. Suppose for contradiction
that $E$ is not an infinite set.  Then, on one hand,
there exists $a\in E$ such that $ba^{-1}\in \Phi\cup \{e_G\}$ for
all $b\in E$. On the other hand,  since $(x,y)\in
U_{r,a}$, there exists a $t(a)\in S\setminus \{e_G\}$
such that $(x,y)\in t(a)aU_r$,  and consequently, $t(a)a\in E$.
 Now, $t(a)=(t(a)a)a^{-1}\in \Phi\cup \{e_G\}$, a
contradiction to the fact that $t(a)\in \Phi^{-1}$. This shows that
$E$ is an infinite set.

Next, consider the sets $S_i=\{ s\in S:
d(s^{-1}x,s^{-1}y)<\frac{2}{i}\}$, $i\in \mathbb{N}$. Since
$(x,y)\in \bigcap_{g\in S}V_{2i,g}$, each $S_i$ is a non-empty set.
If  there exists $s\in \bigcap_{i=1}^\infty S_i$, then
$s^{-1}x=s^{-1}y$, i.e., $x=y$, a contradiction. Hence
$\bigcap_{i=1}^\infty S_i=\emptyset$. This implies that each $S_i$
is an infinite set  because  $S_1\supseteq S_2\supseteq
S_3\cdots$. Now we take $s_1\in S_1$ and $s_i\in
S_i\setminus\{s_1,s_2,\cdots,s_{i-1}\}$ for $i\ge 2$. Let
$S'=\{s_1^{-1},s_2^{-1},\cdots\}$. Then $S'$ is an infinite subset
of $S^{-1}$ and $(x,y)$ is an $S'$-asymptotic pair. This
proves the claim. \medskip

Since $S$ is a $\Phi$-admissible semigroup, $G,S$
satisfy conditions of Lemma \ref{lem}. It follows that there exists
a measurable partition $\mathcal{P}$ of $(X, G,\mu)$ such that
\begin{itemize}
\item $\overline{\mathcal{P}_\Phi(x)}\subseteq W_S(x,G)$ for any $x\in X$, where
$\mathcal{P}_\Phi=\bigvee_{g\in \Phi}g\mathcal{P}$ and
$\mathcal{P}_\Phi(x)$ is the atom of $\mathcal{P}_\Phi$ containing
$x$, and

\item  $\bigcap_{g\in \Phi} g(\widehat{\mathcal{P}_\Phi}\vee P_\mu(G))=P_\mu(G)$.
\end{itemize}

Let $\mathcal{F}=\widehat{\mathcal{P}_\Phi}\vee P_\mu(G)$ and
$\lambda_0=\mu\times_{\mathcal{F}}\mu$.

\medskip
\noindent{\bf Claim 2:}  For every closed (resp. open)
subset $F$ (resp. $U$) of $X\times X$ with $h(F)\supseteq F$ (resp.
$h(U)\subseteq U$) for any $h\in S$, one has $\lambda(F)\geq
\lambda_0(F)$ (resp. $\lambda(U)\leq \lambda_0(U)$).
\medskip

  Since $S$ is a $\Phi$-admissible semigroup, there exists
$(g_n)_{n\ge 1}\subseteq S^{-1}$ such that  $g_i>_\Phi g_{i+1}$ for
each $i\ge 1$ and  for each element $g\in G$, $\#\{ i\in \mathbb{N}:
g_i>_\Phi g\}<+\infty$. Let $\mathcal{F}_n=g_n(\mathcal{F})$ for
$n\in \mathbb{N}$.

 As $\Phi\Phi\subseteq \Phi$,
 \begin{equation}\label{eq-sss-0}u\mathcal{F}\subseteq \mathcal{F}
  \end{equation}
 for any $u\in \Phi$. For any $g\in \Phi$,
 since $\#\{ i\in \mathbb{N}: g_i>_\Phi g\}<+\infty$, we can find $m\in \mathbb{N}$ such that $g>_\Phi g_m$, i.e.,
  $g^{-1}g_m \in \Phi$. Thus $g^{-1}g_m\mathcal{F}\subset \mathcal{F}$, or equivalently,  $\mathcal{F}_m\subset g\mathcal{F}$. Hence
\begin{equation}\label{eq-sss-1}
\bigcap_{n\ge 1}\mathcal{F}_n\subset g\mathcal{F}
\end{equation}
for any $g\in \Phi$.

Note that $g_n\in S^{-1}\subset \Phi\cup \{e_G\}$ and $g_n^{-1}g_{n+k}\in \Phi$ for $n,k\in \mathbb{N}$.  Hence
$(\mathcal{F}_n)_{n\ge 1}$ is a decreasing sequence of sub-$\sigma$-algebras of $\mathcal{F}$, and
$$\bigcap_{n\ge 1}\mathcal{F}_n=\bigcap_{g\in \Phi} g\mathcal{F}=P_\mu(G)$$
by \eqref{eq-sss-0} and \eqref{eq-sss-1}.

Let $\lambda_n=\mu\times_{\mathcal{F}_n}\mu$, $n\in \mathbb{N}$.
Then by \eqref{eq-cdeq},
\begin{equation}\label{eq-fn-1}
\lambda_n=g_n(\lambda_0), \; n\in \mathbb{N}.
\end{equation}
It follows from the Martingale Theorem
that, for any  bounded Borel measurable functions $f,h$ on $X$,
\begin{equation}\label{eq-ccc}
\begin{split}
\int_{X\times X} f(x)h(y) d\lambda_n(x,y)&= \int_X \mathbb{E}_\mu(f|\mathcal{F}_n)(x)\mathbb{E}_\mu(h|\mathcal{F}_n)(x)d\mu(x)\\
&\rightarrow \int_X \mathbb{E}_\mu(f|P_\mu(G))(x)\mathbb{E}_\mu(h|P_\mu(G))(x)d\mu(x) \\
&=\int_{X\times X} f(x)h(y) d\lambda(x,y).
\end{split}
\end{equation}
The family $\mathcal{H}$ of continuous functions $H$ on
$X\times X$ such that
\begin{equation}\label{H1}
\int_{X\times X} H(x, y) d\lambda_n(x,y)\rightarrow \int_{X\times X}
H(x, y) d\lambda(x,y)
\end{equation}
 is a closed subspace of
$C(X\times X)$. By \eqref{eq-ccc}, it contains
the set $\mathcal{H}_*$ of all linear combinations of
functions of the form $f(x)h(y)$ such that $f,g\in
C(X)$. Since $\mathcal{H}_*$ is dense in $C(X\times X)$, $\mathcal{H}=
C(X\times X)$, i.e., \eqref{H1} holds for all $H\in C(X\times X)$,
or equivalently,  $\lambda_n$ converges weakly to $\lambda$ as
$n\rightarrow\infty$.

Now let $F$ be a closed subset of $X\times X$ with
$hF\supseteq F$ for any $h\in S$. Then $g_n^{-1}(F)\supseteq F$ for
$n\in \mathbb{N}$. Since $F$ is closed and $\lambda_n\rightarrow
\lambda$ weakly, we have by \eqref{eq-fn-1} that
$$\lambda (F) \ge  \limsup_{n\rightarrow \infty}
\lambda_n(F)=\limsup_{n\rightarrow \infty} \lambda_0(g_n^{-1}(F))\ge
\lambda_0(F).
$$
This proves the claim because the case of an open
subset $U$ of $X\times X$ simply follows by setting $U=X\setminus
F$.

\medskip

\noindent{\bf Claim 3:} $\lambda_0(W)=1$.
\medskip

 Since $S$ is a semigroup and
$S\setminus \{e_G\}\subseteq \Phi^{-1}$, we have
$$S(S\setminus \{e_G\})\subset S\cap \big (
(\Phi^{-1}\cup\{e_G\})\Phi^{-1} \big)\subset S\cap
\Phi^{-1}=S\setminus \{e_G\},$$ i.e., $S(S\setminus \{e_G\})\subset
S\setminus \{e_G\}$. For given $g\in G$ and
$i\in\mathbb{N}$, by the construction of $U_{i,g}$ and $V_{i,g}$,
we know that
\begin{equation}\label{eq-cc-1}
h(U_{i,g})\subseteq U_{i,g} \text{ and } h(V_{i,g})\subseteq
V_{i,q},\qquad\; h\in S.
\end{equation}  Since $\lambda$ is $G$-invariant,
we have
$$\lambda(h(V_{i,g})\Delta V_{i,g})=0 \text{ and } \lambda(h(U_{i,g})\Delta U_{i,g})=0,\qquad
h\in S.$$  Let
$$G_*:=\{ h\in G:\lambda(h(V_{i,g})\Delta V_{i,g})=0 \text{ and } \lambda(h(U_{i,g})\Delta U_{i,g})=0\}.$$
Then $G_*$ is a subgroup of $G$ and $S\subset G_*$. It
follows from the fact  $<S>=G$ that $G_*=G$. Thus,
$$\lambda(h(V_{i,g})\Delta V_{i,g})=0 \text{ and }\lambda(h(U_{i,g})\Delta U_{i,g})=0,\qquad\;
h\in G.
$$  By the
ergodicity of $\lambda$ and noting that
$\lambda(U_{i,g})\ge \lambda(U_i)>0$ and $\lambda(U_{i,g})\ge
\lambda(V_i)>0$, we have
\begin{equation}\label{eq-cc-2}
\lambda(U_{i,g})=\lambda(V_{i,g})=1.
\end{equation}
Combining Claim 2 with \eqref{eq-cc-1} and
\eqref{eq-cc-2}, we have
$\lambda_0(U_{i,g})=\lambda_0(U_{i,g})=1$. Since $g,i$
are arbitrary, the claim follows.
\medskip

To proceed with the proof of the theorem, we let
$\mu=\int_X \mu_x d\mu(x)$ be the disintegration of $\mu$ over
$\mathcal{F}$. Then by Claim 3,
$$1=\lambda_0(W)=\int_X \mu_x\times\mu_x(W) d\mu(x).$$ Thus for $\mu$-a.e. $x\in X$,
\begin{align}\label{ae=1}
\mu_x\times \mu_x(W)=1.
\end{align}

Since $h_\mu(G)>0$, we have by Lemma
\ref{bl} 2) that $\lambda(\Delta_X)=0$. Note that
$h(\Delta_X)=\Delta_X$, $h\in S$. We then have by Claim
2 that $\lambda_0(\Delta_X)\le \lambda(\Delta_X)=0$. Thus,
$$\int_X \mu_x\times \mu_x(\Delta_X)d\mu(x)=\lambda_0(\Delta_X)=0,$$
i.e., for $\mu$-a.e.$x\in X$, $\mu_x\times
\mu_x(\Delta_X) =0$. This implies that $\mu_x$ is non-atomic  for
$\mu$-a.e. $x\in X$. Combing this with \eqref{ae=1}, we have that
for $\mu$-a.e. $x\in X$, $W\cap supp(\mu_x)\times supp(\mu_x)$ is  a
dense $G_\delta$ subset of $supp(\mu_x)\times supp(\mu_x)$ and
$supp(\mu_x)$ is a perfect set. By Lemma \ref{Myc} and
Claim 1, we have that for $\mu$-a.e. $x\in X$, $supp(\mu_x)$
contains a $(S^{-1},\delta)$-Li-Yorke set which is a union of
countably many Cantor sets.  But since
$\mu_x(\mathcal{P}_\Phi(x))=1$ for $\mu$-a.e.$x\in X$,
$supp(\mu_x)\subseteq \overline{\mathcal{P}_\Phi(x)}\subseteq
W_S(x,G)$ for $\mu$-a.e.$x\in X$. In other words, for
$\mu$-a.e. $x\in X$, $W_S(x,G)$ contains a
$(S^{-1},\delta)$-Li-Yorke set which is a union of countably many
Cantor sets. This  completes the proof.
\end{proof}


\end{document}